\documentclass{article}
\usepackage[utf8]{inputenc}
\usepackage{amsmath, amsthm, amssymb, amsfonts}
\usepackage{bm}
\usepackage{dsfont}
\usepackage[utf8]{inputenc}
\usepackage{rotate}
\usepackage{tikz}
\usepackage{tikz-cd}
\usepackage[arrow,matrix,curve]{xy} 	
\usepackage{xcolor}
\usepackage{todonotes}
\usepackage{alltt}
\usepackage{amsmath,amssymb}
\usepackage{comment}

\theoremstyle{plain}
\newtheorem{theorem}{Theorem}[section]

\newtheorem{lemma}[theorem]{Lemma}

\theoremstyle{definition}

\newtheorem{remark}[theorem]{Remark}

\newcommand{\norm}[1]{\left\lVert#1\right\rVert}

\title{Linear Bounds for Differentiable Limits of Weak Pair Correlation Functions}
\author{Christian Weiss}
\date{\today}

\begin{document}

\maketitle

\begin{abstract} For $s \geq 0$ and a parameter $0 < \beta < 1$, the weak pair correlation function $f_{N,\beta}(s)$ for the first $N \in \mathbb{N}$ elements of a sequence $(x_n)_{n \in \mathbb{N}} \subset[0,1]$ is evidently non-decreasing in $s$. Moreover, it satisfies $\lim_{N \to \infty} f_{N,\beta}(0) = 0$ if the elements of $(x_n)_{n \in \mathbb{N}}$ are distinct. Beyond these basic observations, little is known in general about the behavior of the limiting function. In this note, we investigate the situation in which the limit $f_\beta(s)=\lim_{N\to\infty} f_{N,\beta}(s)$ exists for all $s\ge 0$ and is differentiable in a neighborhood of the origin. Under these assumptions, we establish the bounds $2s \le f_\beta(s) \le f'_\beta(0)\, s,$ thereby providing general constraints on the limiting function.
\end{abstract}

\section{Introduction}

A sequence of real numbers $(x_n)_{n \in \mathbb{N}} \subset [0,1)$ is said to have weak Poissonian pair correlations with parameter $0 \leq \beta \leq 1$, if
\[
F_{N,\beta}(s) = \frac{1}{N^{2-\beta}} \left\{ 1 \leq i \neq j \leq N \, : \, \norm{x_i-x_j} \leq \frac{s}{N^\beta} \right\},
\]
where $\norm{\cdot}$ denotes the distance to the nearest integer and $N \in \mathbb{N}$, converges to $2s$ as $N \to \infty$ for all $s \geq 0$. If $\beta = 1$, we just speak of Poissonian pair correlations. Weak Poissonian pair correlations are a generic property of uniformly distributed random variables: if $(X_n)_{n \in \mathbb{N}}$ is a sequence of random variables, independently drawn from uniform distribution, then $(X_n)_{n \in \mathbb{N}}$ almost surely has weak Poissonian pair correlations for all $0 \leq \beta \leq 1$. While finding \textit{deterministic} sequences possessing Poissonian pair correlations is a difficult task and only relatively few classes of example have been found, see e.g. \cite{LST24} for a more recent one, it is a lot easier to find examples if $\beta < 1$ by exploiting connections to discrepancy theory, see e.g. \cite{Wei21b}.\\[12pt]
In the non-generic case, if $\lim_{N \to \infty} F_{N,\beta}(s)= f_\beta(s)$ exists for all $s \geq 0$ but $f_\beta(s) \neq 2s$, much less known. If the limiting function $\lim_{N \to \infty} F_{N,\beta}(s)=f_\beta(s)$ exists for all $s\geq 0$ and is continuous, then we say that the sequence $(x_n)_{n \in \mathbb{N}}$ has $(f_\beta,\beta)$-pair correlations, compare \cite{FW25}. It is trivial that $f_\beta(s)$ must be a non-decreasing function. In addition, $f_\beta(0) = 0$ is automatic if the number of multiply appearing elements $n_\beta(N)$ is of order $o\left(N^{1-\tfrac{\beta}{2}}\right)$. This happens for instance if the elements of $(x_n)_{n\in\mathbb{N}}$ are distinct, which imposes only a rather mild restriction on the set of admissible sequences $(x_n)_{n \in \mathbb{N}}$. Also in the cases $n_\beta(N) \asymp c N^{1-\tfrac{\beta}{2}}$ and $n_\beta(N) \gg N^{1-\tfrac{\beta}{2}}$, it is not hard to calculate $f_\beta(0)$ explicitly. Therefore, it is not a great restriction to concentrate on sequences with $f_\beta(0) = 0$ in the following.\\[12pt]
Apart from the two mentioned simple facts, only little is known about general properties of possible limiting functions $f_\beta(s)$. In the case $\beta=1$, it has been shown in \cite[Theorem~1.2]{FW25} that $f_1(s)$ cannot be continuous if the number of different gap lengths, i.e. the distances between geometrically neighboring elements of the finite sequence $x_1,\ldots,x_N$, is uniformly bounded. If $\beta < 1$, this property does not hold as follows e.g. from \cite{Wei21b}. Although there are few examples of explicitly known limiting functions $f_\beta(s)$ with $f_\beta(s) \neq 2s$ that have been discussed in the literature, see \cite{ALP18, Lut20, MS13, Say23, Wei23}, no non-trivial general property of $f_\beta(s)$ has been identified for $\beta < 1$, to the best of the author's knowledge. Nonetheless, it is striking that for $\beta < 1$ all known examples satisfy $f_\beta(s) \geq 2s$, while this is not necessarily the case for $\beta = 1$.\\[12pt]
In this note, we show that the observation $f_\beta(s) \geq 2s$ for $\beta < 1$ is not merely an artifact of the limited number of known examples, but in fact reflects a general property of $f_\beta(s)$. More surprisingly, $f_\beta(s)$ can also be bounded linearly from above if $f_\beta(s)$ is differentiable close to $0$. Indeed, our main result is the following.
\begin{theorem} \label{thm:main} Let $(x_n)_{n \in \mathbb{N}} \subset [0,1]$ and $0 \leq \beta < 1$. If there is a $s_0 >0$ so that the limit
\[
\lim_{N \to \infty} F_{N,\beta}(s) = f_\beta(s)
\]
exists for all $s < s_0$ with $f_\beta(0)=0$ and $f_{\beta}(s)$ is differentiable in this range, then it necessarily holds that
\[
2s \leq f_{\beta}(s) \leq f_\beta'(0)s
\]
for all $s \geq 0$ for which the limit exists. If the limit does not exist, we still have
\[
2s \leq \liminf_{N \to \infty} F_{N,\beta}(s) \leq \limsup_{N \to \infty} F_{N,\beta}(s)  \leq f_\beta'(0)s.
\]
\end{theorem}
Theorem~\ref{thm:main} may be seen as a generalization of \cite[Theorem~2]{HZ24}, where it was shown that $\lim_{N \to \infty} F_{N,\beta}(s) = 2s$ for $s < s_0$ implies that $(x_n)_{n \in \mathbb{N}}$ has weak Poissonian pair correlations for $\beta$. Moreover, the result shows that the pair correlation statistic in a neighborhood of $0$ already determines its behavior on the entire real line.\\[12pt]
From \cite[Theorem~4]{HZ24} it follows that weak Poissonian pair correlations with parameter $\beta$ imply weak Poissonian pair correlations with parameter $\alpha$ for $0 \leq \alpha < \beta \leq 1$. It is now natural to ask if a similar property holds true for continuous limits $f_\beta(s)$ of the pair correlation statistic $F_{N,\beta}(s)$. Indeed, we show in Theorem~\ref{thm:monotone} for $0 \leq \alpha < \beta < 1$ that necessarily $f_{\alpha}(s)\leq f_{\beta}(s)$ for all $s \geq 0$ if the two limiting function exist. In the special case $f_{\beta}(s)=2s$, it is not difficult to see that this result implies \cite[Theorem~4]{HZ24}. Having this observation in mind, it might come as a surprise that the monotonicity of functions is not anymore true for $\beta = 1$ as the example from \cite{Wei23} shows. The reason why this happens will become visible from the proof of Theorem~\ref{thm:monotone}

\section{Proof of the results}
We start by mentioning two properties which are equivalent to $(f_\beta,\beta)$-pair correlations. Despite that the proof of the following lemma is in large parts in parallel to \cite[Lemma~9]{HZ24}, some additional arguments are required for continuous $f(s) \neq 2s$ (we do not assume that $f$ is differentiable here). Therefore, we include the proof in this note for the sake of completeness. For brevity of notation, we moreover define 
    \[
    I_{N,\beta}(s) := \int_0^1 F_{N,\beta}(t,s)^2 \mathrm{d}t,
    \]
where
    \[
    F_{N,\beta}(t,s) := \frac{1}{N^{1-\beta}} \left\{ 1 \leq n \leq N \, : \, \norm{x_n-t}\leq \frac{s}{N^\beta} \right\}.
    \]
for $t \in \mathbb{R}$. Note that $F_{N,\beta}(t,s)$ is $1$-periodic in $t$.
\begin{lemma} \label{lem:equivalence} Let $(x_n)_{n \in \mathbb{N}} \subset [0,1]$ be a sequence, $s_0 >0$, and $0 \leq \beta \leq 1$ and $f_\beta: \mathbb{R}^+_0 \to \mathbb{R}^+_0$ be a continuous function. Then the following are equivalent.
\begin{itemize}
    \item[(i)] The $\beta$-pair correlation function satisfies 
    \[
    \lim_{N \to \infty} F_{N,\beta}(s) = f_\beta(s), \qquad \mathrm{for\  all\ } s < s_0.
    \]
    \item[(ii)] The following equality holds
    \[
    \lim_{N \to \infty} \int_0^s F_{N,\beta}(t) \mathrm{d}t = \int_0^s f_\beta(t) \mathrm{d}t \qquad \mathrm{for\  all\ } s < s_0.
    \]
    \item[(iii)] The integral $I_{\beta,N}(s)$ satisfies
    \[
    \lim_{N \to \infty} I_{N,\beta}(s) = \begin{cases} \int_0^s f_\beta(t) \mathrm{d}t + s & \beta = 1\\ 
    \int_0^s f_\beta(t) \mathrm{d}t & \beta < 1
    \end{cases}
    \qquad \mathrm{for\  all\ } s < s_0.
    \]
\end{itemize}
The equivalence of three analogous properties also holds true if $\lim_{N \to \infty}$ and $=$ in each of the equations are either replaced by $\limsup_{N \to \infty}$ and $\leq$ or $\liminf_{N \to \infty}$ and $\geq$. 
\end{lemma}
\begin{proof} First, we show the equivalence of (i) and (ii). Hence assume that (i) is true and let $s < s_0$ be arbitrary. Choose $M \geq 1$ and consider $0 \leq j \leq M-1$. Since $F_{N,\beta}(t)$ is monotonic in $t$, it holds that
\[
\frac{s}{M}F_{N,\beta}\left( \frac{js}{M} \right) \leq \int_{j/M}^{(j+1)/M} F_{N,\beta}\left( t \right) \textrm{d} t \leq \frac{s}{M}F_{N,\beta}\left( \frac{(j+1)s}{M} \right) 
\]
and summing up yields
\[
\frac{s}{M}\sum_{j=0}^{M-1} F_{N,\beta}\left( \frac{js}{M} \right) \leq \int_{0}^{s/M} F_{N,\beta}\left( t \right) \textrm{d} t \leq \frac{s}{M} \sum_{j=1}^{M}F_{N,\beta}\left( \frac{js}{M} \right) 
\]
Letting $N \to \infty$ (and assuming the existence of the limit in the middle term for a second) it follows that
\[
\frac{s}{M} \sum_{j=0}^{M-1} f_\beta\left(\frac{js}{M}\right) \leq \lim_{N \to \infty} \int_{0}^{s/M} F_{N,\beta}\left( t \right) \textrm{d} t \leq \frac{s}{M} \sum_{j=1}^M f_\beta\left( \frac{js}{M} \right)
\]
Since $f_\beta$ is continuous, it is integrable on $[0,s]$. Moreover, the sum on the right as well as the one on the left are Riemann sums and thus converge to $\int_0^sf(t)\mathrm{d}t$ as claimed.\\[12pt]
If (ii) holds, let $\varepsilon > 0$ be arbitrary with $s + \varepsilon < s_0$. Using the monotonicity of $F_{N,\beta}(s)$ again, we obtain
\[
\frac{1}{\varepsilon} \int_{s - \varepsilon}^s F_{N,\beta}(t)\mathrm{d}t \leq F_{N,\beta}(s) \leq \frac{1}{\varepsilon} \int_{s}^{s+\varepsilon} F_{N,\beta}(t)\mathrm{d}t.
\]
Taking the limit for $N \to \infty$ and applying (ii), it follows that
\[
\frac{1}{\varepsilon} \int_{s-\varepsilon}^{s} f_\beta(t)\mathrm{d}t \leq \lim_{N \to \infty} F_{N,\beta}(s) \leq \frac{1}{\varepsilon} \int_{s}^{s+\varepsilon} f_\beta(t)\mathrm{d}t
\]
and the terms on the left and on the right converge to $f_\beta(s)$ as $\varepsilon \to 0$ by the Fundamental Theorem of Calculus (or the mean value theorem for integrals). Hence, the claim follows.\\[12pt]
Next we show that (ii) and (iii) are equivalent. For this purpose, let $\lambda(\cdot)$ denote the 1-dimensional Lebesgue measure and $B(x,r)$ be the ball of radius $r$ centered at $x$. Then $I_{N,\beta}(s)$ equals
\begin{align*}
I_{N,\beta}(s) & = \frac{1}{N^{2-2\beta}} \sum_{i,j=1}^N\lambda \left( B\left(x_i,\frac{s}{2N^\beta}\right) \cap B\left(x_i,\frac{s}{2N^\beta}\right)  \right)\\
&=\frac{1}{N^{2-2\beta}} \sum_{\substack{i,j=1\\i\neq j}}^N\lambda \left( B\left(x_i,\frac{s}{2N^\beta}\right) \cap B\left(x_i,\frac{s}{2N^\beta}\right)  \right) + \frac{1}{N^{2-2\beta}} \cdot N \cdot \frac{s}{N^\beta}\\
& = \frac{s}{N^{2-\beta} } \sum_{\substack{i,j=1\\i\neq j}}^N \max\left(1 - \frac{\norm{x_i-x_j}}{s/N^\beta} ,0\right) + \frac{s}{N^{1-\beta}}
\end{align*}
and 
\begin{align*}
    \int_0^s F_{N,\beta}(t) \mathrm{d}t & = \int_0^s \frac{1}{N^{2-\beta}} \sum_{\substack{i,j=1\\i\neq j}}^N \mathds{1}_{B(0,t/N^\beta)}(x_i-x_j) \mathrm{d}{t}\\
    & = \frac{1}{N^{2-\beta}} \sum_{\substack{i,j=1\\i\neq j}}^N \int_0^1 \mathds{1}_{[\norm{x_i-x_i}N^\beta,\infty)}(t) \mathrm{d}t\\
    & = \frac{s}{N^{2-\beta}} \sum_{\substack{i,j=1\\i\neq j}}^N \max\left(1 - \frac{\norm{x_i-x_j}}{s/N^\beta} ,0\right)
\end{align*}
which implies the equivalence of (ii) and (iii) as $N \to \infty$.\\[12pt]
Finally, note that each equality in the steps above may be replaced by $\leq$ if we use $\limsup_{N \to \infty}$ instead of $\lim_{N \to \infty}$. Similarly, moving to $\liminf_{N \to \infty}$ requires to use $\geq$ instead of $=$.
\end{proof}
Based on the lemma, we can now use the equivalent properties (i), (ii) and (iii) in parallel to show Theorem~\ref{thm:main}. In the proof, it will turn out to be essential that $f$ is indeed differentiable close to $0$.
\begin{proof}
According to the assumptions there is some differentiable function $f_\beta$ with
\[
    \lim_{N \to \infty} F_{N,\beta}(t) = f_\beta(t), \qquad \forall t < s_0.
\]
Now let $s>0$ be arbitrary and choose $K>0$ large enough such that $s/K < s_0$. Then we have the inclusion
\[
B\left(t,\frac{s}{2N^\beta}\right) \subset \bigcup_{|l| \leq K/2} B\left(t+\frac{ls}{KN^\beta},\frac{s}{2KN^\beta}\right)
\]
and by applying the Cauchy-Schwartz inequality and (iii) in Lemma~\ref{lem:equivalence} it thus follows that
\begin{align*}
\int_0^1 F_{N,\beta}(t,s)^2 \mathrm{d}t &\leq \int_0^1 \left( \sum_{|l| \leq K/2}F_{N,\beta}\left(t+\frac{ls}{KN^\beta},\frac{s}{K}\right)\right)^2 \mathrm{d}t \\
&\leq \int_0^1 (K+1)  \sum_{|l| \leq K/2}F_{N,\beta}\left(t+\frac{ls}{KN^\beta},\frac{s}{K}\right)^2 \mathrm{d}{t} \\
& = (K+1)^2 \int_0^1 F_{N,\beta}\left(t, \frac{s}{K} \right)^2 \mathrm{d}t\\
& = (K+1)^2 \int_0^{s/K} f_\beta(t)\mathrm{d}t.
\end{align*}
As $f_\beta(0)=0$ and $f_\beta$ is differentiable in a neighborhood of $0$, the limit of the integral on the right hand side can for $K \to \infty$ be evaluated as 
\begin{align*}
\lim_{K \to \infty} (K+1)^2 \int_0^{s/K} f_\beta(t)\mathrm{d}t & = \lim_{K \to \infty}(K+1)^2 \int_{0}^{s/K} (f_\beta'(0)t + o(t)) \mathrm{d}t\\
&=  \lim_{K \to \infty} (K+1)^2 f_\beta'(0) \frac{s^2}{2K^2} + \lim_{K \to \infty} (K+1)^2 o\left(\frac{1}{K^2}\right)\\
& = \frac{f_\beta'(0)}{2} s^2.
\end{align*}
By once again applying the Cauchy-Schwartz inequality, the well-known lower bouund
\[
I_{N,\beta}(s) = \int_0^1 F_{N,\beta}(t,s)^2\mathrm{d}t \geq \left( \int_0^1 F_{N,\beta}(t,s)\mathrm{d}t \right)^2 = s^2
\]
follows. Using the equivalence of (ii) in (iii) in Lemma~\ref{lem:equivalence}, we arrive at
\begin{align} \label{ineq:integral}
s^2 \leq \int_0^s f_\beta(t)\mathrm{d}t \leq \frac{f_\beta'(0)}{2}s^2.
\end{align}
Finally, we apply the equivalent property (i) in Lemma~\ref{lem:equivalence} which yields the claim on the limit. As a corresponding statement holds for $\liminf$ and $\limsup$ instead of $\lim$ according to Lemma~\ref{lem:equivalence}, also 
\[
2s \leq \liminf_{N \to \infty} F_{N,\beta}(s) \leq \limsup_{N \to \infty} F_{N,\beta}(s)  \leq f_\beta'(0)s.
\]
follows. 
\end{proof}

\begin{remark} One might be tempted to assume that \eqref{ineq:integral} directly implies the inequality on $f_\beta(s)$ without referring to Lemma~\ref{lem:equivalence}. However, we would then get an additional factor of $1/2$ and $2$ respectively in the inequalities as the following example shows: Let the integral $\int_0^s f(t)\mathrm{d}t$ of an arbitrary function $f$ be bounded from above by $cs^2$. Hence $f(s)=2cs$ for $0\leq s\leq 1$ and $f(s)=c$ for $1< s \leq (c/2+\delta)$ for $\delta>0$ is a function compatible with the upper bound since $\int_0^{c/2+\delta} f(t)\mathrm{d}t=c^2 + \left(\frac{c}{2}+\delta-1\right)c = \frac{3}{2}c^2 + c\delta-c$, but $f(c/2+\delta)=c<c+2\delta$. At the same time $f(s)$ is compatible with the lower bound $\int_0^s f(t)\mathrm{d}t > s^2$ if $c>\sqrt{2}$ as long as $\delta > 0$ is small enough so that
\[
(c+\delta)^2 = c^2 + 2c\delta + \delta^2 < \frac{3}{2}c^2 + c\delta - 1
\]
A similar example can be constructed to also show that the upper bound cannot be improved by a factor of $2$ in general.\end{remark}

\paragraph{Monotonicity of the Pair Correlation Statistic in $\beta$} According to \cite[Theorem~4]{HZ24} weak Poissonian pair correlations with parameter $\beta$ imply weak Poissonian pair correlations with parameter $\alpha$ for $0 \leq \alpha < \beta \leq 1$. For continuous limits we show here the seemingly slightly weaker property, that if a sequence has $(f_\beta,\beta)$ and $(g_\alpha,\alpha)$ pair correlations with $0 \leq \alpha < \beta \leq 1$, then $g_\alpha(s) \leq f_\beta(s)$ for all $s \geq 0$. However, note that \cite[Theorem~4]{HZ24} is a consequence of this result for $\beta < 1$ and hence Theorem~\ref{thm:monotone} may be regarded as its generalization.
\begin{theorem} \label{thm:monotone} Let $(x_n)_{n \in \mathbb{N}} \subset [0,1]$ and let $0 \leq \alpha < \beta < 1$. Furthermore assume that $(x_n)_{n \in \mathbb{N}}$ has $(f_\beta,\beta)$- as well as $(g_\alpha,\alpha)$-pair correlations. Then it holds for all $s \geq 0$ that
\[
g_\alpha(s) \leq f_\beta(s).
\]    
Without assuming that $\lim_{N \to \infty} F_{N,\alpha}(s)$ exists, it still holds that
\[
\limsup_{N \to \infty} F_{N,\alpha}(s) \leq f_\beta(s).
\]
\end{theorem}
From the proof of Theorem~\ref{thm:monotone} it becomes visible that its statement fails for $\beta=1$ because of the additional factor $s$ on the right hand side  of Lemma~\ref{lem:equivalence} (iii).
\begin{proof} Let
\[
M = \left\lceil \frac{N^{\beta-\alpha}}{2} \right\rceil
\]
and recall that 
\[
B\left(t,\frac{s}{2N^\alpha}\right) \subset \bigcup_{|l| \leq M} B\left(t+\frac{ls}{N^\beta},\frac{s}{2N^\beta}\right).
\]
We now apply the Cauchy-Schwartz inequality to get
\begin{align*}
\int_0^1 F_{N,\alpha}(t,s)^2 \mathrm{d}t &\leq \int_0^1 \frac{1}{N^{2(\beta-\alpha)}}\left( \sum_{|l| \leq M}F_{N,\beta}\left(t+\frac{ls}{N^\beta},s\right)\right)^2 \mathrm{d}t \\
&\leq \int_0^1 \frac{N^{\beta-\alpha}+1}{N^{2(\beta-\alpha)}}  \sum_{|l| \leq M}F_{N,\beta}\left(t+\frac{ls}{N^\beta},s\right)^2 \mathrm{d}{t} \\
& =  \frac{N^{\beta-\alpha}+1}{N^{2(\beta-\alpha)}}  \sum_{|l| \leq M} \int_0^1 F_{N,\beta}\left(t, s \right)^2 \mathrm{d}t\\
& \leq  \left( 1 + \frac{3}{N^{\beta-\alpha}}\right) \int_0^{1} F_{N,\beta}\left(t, s \right)^2 \mathrm{d}t.
\end{align*}
Letting $N \to \infty$ and using Lemma~\ref{lem:equivalence} (iii) implies
\[
\int_0^s g_\alpha(t)\mathrm{d}t \leq \int_0^s f_{\beta}(t)\mathrm{d}t.
\]
Hence, the claim follows from Lemma~\ref{lem:equivalence} (i). The claim on the limit superior follows by replacing $\lim$ by $\limsup$ in the lines above.
\end{proof}

\bibliographystyle{alpha}
\bibdata{literatur}
\bibliography{literatur}

\textsc{Ruhr West University of Applied Sciences, Duisburger Str. 100, D-45479 M\"ulheim an der Ruhr,} \texttt{christian.weiss@hs-ruhrwest.de}

\end{document}